\documentclass{amsart}
\usepackage{amsfonts,amssymb,amsmath,amsthm}
\usepackage{url}
\usepackage{enumerate}
\usepackage{verbatim}

\newtheorem{theorem}{Theorem}[section]
\newtheorem{lemma}[theorem]{Lemma}
\newtheorem{proposition}[theorem]{Proposition}

\theoremstyle{definition}

\newtheorem{remark}[theorem]{Remark}

\newcommand{\litem}[1][\vrule width0pt]%
{\qquad\quad\begin{minipage}{2mm} \item #1 \end{minipage}}

\newcommand{\rk}{\operatorname{rk}}
\newcommand \Ker{\mathrm{Ker}\,}
\def\id{\operatorname{id}}      

\def\Ad{\operatorname{Ad}}
\def\ad{\operatorname{ad}}      
\def\dim{\operatorname{dim}}      
\def\End{\operatorname{End}}        
\def\Span{\operatorname{Span}}      
\DeclareMathOperator{\Img}{Im}
\def\ve{\varepsilon}
\def\R{\mathbb R} 

\def\ag{\mathfrak a}
\def\g{\mathfrak g}
\def\h{\mathfrak h}
\def\H{\mathfrak H}
\def\m{\mathfrak m}
\def\n{\mathfrak n}
\def\so{\mathfrak{so}}
\def\sl{\mathfrak{sl}}
\def\ve{\mathfrak v}
\def\va{\varphi}
\def\z{\mathfrak z}
\def\<{\langle}                     
\def\>{\rangle}
\def\ip{\<\cdot,\cdot\>}

\def\a{\alpha}
\def\b{\beta}
\def\la{\lambda}
\def \ve{\varepsilon}
\DeclareMathOperator{\Tr}{Tr \,}

\def\t{\mathfrak t}

\begin{document}

\author[Cairns]{Grant Cairns}
\address{Department of Mathematics and Statistics, La Trobe University, Melbourne, Australia 3086}
\email{G.Cairns@latrobe.edu.au}

\author[Hini\'c Gali\'c]{Ana Hini\'c Gali\'c}
\email{A.HinicGalic@latrobe.edu.au}

\author[Nikolayevsky]{Yuri Nikolayevsky}
\email{Y.Nikolayevsky@latrobe.edu.au}

\author[Tsartsaflis]{Ioannis Tsartsaflis}
\email{itsartsaflis@students.latrobe.edu.au}


\thanks{This project was supported by ARC Discovery Grant DP130103485}

\title[geodesics]{Geodesic Bases for Lie Algebras}

\begin{abstract}
For finite dimensional real Lie algebras, we investigate the existence of an inner product having a basis comprised of geodesic elements. We give several existence and non-existence results in certain cases: unimodular solvable Lie algebras having an abelian nilradical, algebras having an abelian derived algebra, algebras having a codimension one ideal  of a particular kind, nonunimodular algebras of dimension $\leq 4$, and unimodular algebras of dimension $5$. \end{abstract}

\maketitle

\section{Introduction}

 Consider a finite dimensional real Lie group $G$ with Lie algebra $\g$. For a given inner product $\ip$ on $\g$,  a nonzero element $X\in\g$ is a \emph{geodesic element} if the corresponding left-invariant vector field on $G$ is a geodesic element field, relative to the left-invariant Riemannian metric on $G$ determined by $\ip$; see \cite{Ka,CKM,KNV,KS,CHGN1,CHGN2}. 
A simple algebraic condition for $X$ to be a geodesic element is recalled in Section~\ref{S:prelim}. For a given inner product on $\g$, we will say that a basis $\{X_1,\dots,X_n\}$ for $\g$ is a \emph{geodesic basis} if each of the elements $X_i$ is a geodesic element.

It is easy to see that if a  Lie algebra $\g$ possesses an inner product with an orthonormal geodesic basis, then $\g$ is necessarily unimodular. All semisimple Lie algebras have an inner product  with an orthonormal geodesic basis \cite{KS}, and so too do all nilpotent Lie algebras \cite{CLNN}. It was proved in \cite{CLNN} that for every unimodular Lie algebra of dimension $\leq 4$, every inner product  has an orthonormal geodesic basis.  An example was given in \cite{CLNN} of  a 5-dimensional unimodular Lie algebra that has no orthonormal geodesic basis for any inner product; nevertheless this algebra does have a (nonorthonormal) geodesic basis for a certain inner product. This raises three natural open questions (the first one was posed in \cite{CLNN}):
\begin{enumerate}
\item[1.] Does every unimodular Lie algebra possess an inner product having a geodesic basis?
\item[2.] Which unimodular Lie algebras possess an inner product having an orthonormal geodesic basis?
\item[3.] Which nonunimodular Lie algebras possess an inner product having a  geodesic basis?
\end{enumerate}

The present paper aims to give further results, answering the above questions in certain cases. Our main results concern the following 5 cases:
\begin{itemize}
\item  Unimodular solvable Lie algebras having an abelian nilradical.
\item  Certain Lie algebras having an abelian derived algebra.
\item  Lie algebras having a codimension one ideal  of a particular kind.\item  Non-unimodular Lie algebras of dimension $\leq 4$.
\item Unimodular Lie algebras of dimension $5$.
\end{itemize}

Let us state these main results. 

\begin{theorem}\label{th:nilabel}
If $\g$ is a unimodular solvable Lie algebra with an abelian nilradical $\n$, then there exists an inner product on $\g$ having a geodesic basis.
\end{theorem}

Recall that given a Lie algebra $\h$ and a derivation $\varphi$ of $\h$, a Lie algebra $\h_\varphi$ of dimension $\dim(\h)+1$ is defined by retaining the structure on $\h$, introducing a new element $X$ and defining $[X,Y]:=\varphi(Y)$ for all $Y\in\h$; see \cite{Dix}. The algebra  $\h_\varphi$ might sensibly be called the \emph{suspension} of $\varphi$, by analogy with the classic construction in topology.

Consider the abelian algebra $\R^{n}$, and identity map $\id:\R^{n}\to \R^{n}$. Let $\mathcal{A}_n$ denote the algebra $\R^{n}_{\id}$. It is easy to see that there is no inner product on $\mathcal{A}_n$ with a geodesic basis; see Lemma \ref{Lemma1} below.

\begin{theorem}\label{th:rdiag}
Suppose $\g$ is a Lie algebra with an  abelian derived algebra $\g'$, and  that for all $Y \in \g$, the restriction $(\ad(Y))_{|\g'}$ is semisimple, with real eigenvalues. Then $\g$ admits an inner product with a geodesic basis  unless $\g$ is isomorphic to $\mathcal{A}_n$ for some $n\geq 1$.
\end{theorem}

\begin{theorem}\label{th:oneabel}
Suppose a Lie algebra $\g$ has a codimension one abelian ideal. Then
$\g$ admits an inner product with a geodesic basis  unless $\g$ is isomorphic to $\mathcal{A}_n$ for some $n\geq 1$.
\end{theorem}

We will also require the \emph{Heisenberg Lie algebra} $\H_{2m+1}$ of dimension $2m+1$, which has basis $\{X_1,\dots,X_{2m+1}\}$ and relations $[X_i,X_{i+m}]=X_{2m+1}$, for $ i=1, \dots, m$. 
\begin{theorem}\label{th:oneheis}
Suppose a Lie algebra $\g$ has a codimension one ideal isomorphic to $\H_{2m+1}$. Then
\begin{enumerate}[{\rm (a)}]
  \item
  If $\g$ is nonunimodular, then no inner product on $\g$ has a geodesic basis.

  \item
  If $\g$ is unimodular, then every inner product on $\g$ has an orthonormal geodesic basis. 
\end{enumerate}
\end{theorem}

In every nonunimodular Lie algebra, there is a unique codimension one unimodular ideal called the \emph{unimodular kernel} (see \cite{Mi}); this is just the kernel of the map $\Tr \circ \ad$. So nonunimodular Lie algebras are of the form  $\h_\varphi$, where $\h$ is the unimodular kernel and the derivation $\varphi$ has nonzero trace.
In  \cite[Theorem~1]{CLNN}, it was shown that every unimodular Lie algebra of dimension 4 has an inner product with an orthonormal geodesic basis.

\begin{theorem}\label{T:dim4}
If $\g $ is a nonunimodular Lie algebra of dimension $\leq 4$, then there is an inner product on $\g$ with a geodesic basis unless either 
\begin{enumerate}[{\rm (a)}]
\item $\g\cong\mathcal{A}_n$ for  $n\in\{1,2,3\}$, or
\item  the unimodular kernel of $\g$ is isomorphic to the Heisenberg Lie algebra  $\H_3$, or
\item the unimodular kernel $\h$ of $\g$ is isomorphic to the Lie algebra of the group of isometries of the Euclidean plane, which has basis $\{X_1,X_2,X_3\}$ and relations $[X_1,X_2]=-X_3$, $[X_1,X_3]=X_2$.\end{enumerate}

\end{theorem}

\begin{theorem}\label{T:dim5}
Every 5-dimensional unimodular Lie algebra $\g$ possesses an inner product with a geodesic basis. Moreover, $\g$ has  an inner product with an orthonormal geodesic basis if and only if 
\begin{enumerate}[{\rm (a)}]
\item $\g$ has a codimension one abelian ideal, or
\item $\g$ has a nontrivial centre, or
\item $\g$ has a  three-dimensional abelian nilradical $\n$, and a basis $\{X_1,\dots,X_5\}$ with $\n=\Span(X_3,X_4,X_5)$ so that relative to the basis $\{X_3,X_4,X_5\}$ for $\n$,
\[
\ad(X_1)|_{\n}=\begin{pmatrix} -2 & 0 & 0 \\ 0 & 1&0  \\ 0 & 0 & 1 \end{pmatrix},
\qquad\ad(X_2)|_{\n}=\begin{pmatrix} 0& 0 & 0 \\ 0 &  0 &1  \\ 0 & -1 & 0 \end{pmatrix}.\]
\end{enumerate}
\end{theorem}

\section{Preliminary remarks}\label{S:prelim}

Recall that for an inner product $\ip$ on a Lie algebra $\g$, a nonzero element $X\in \g$ is said to be a \emph{geodesic element},  and to be \emph{geodesic}, if
\begin{equation}\label{eq:def}
    \<X,[X,Y]\>=0
\end{equation}
for all $Y\in\g$. In other words, a nonzero element $X$ is geodesic if and only if $X$ is perpendicular to the image of the adjoint map $\ad(X):\g\to\g, Y\mapsto [X,Y]$. In particular, every nonzero vector from the centre $\z$ is geodesic. Similarly, every nonzero element of the orthogonal complement  $\g'^\perp$ of the derived algebra $\g'$, is geodesic.
For further information about geodesic elements, see
\cite{KNV,KS,Mc,CHGN1,CHGN2,CLNN,Ka,CKM}. 
 Note that geodesic elements are sometimes called \emph{homogeneous geodesics} in the literature.

We start with the following general fact concerning geodesic elements; see \cite[Proposition]{Sze}, \cite[Th\'{e}or\`{e}me~3]{Arn}, cf.~\cite{RS}. 
In fact, we will not require this result in this paper, but we include it for the reader's interest, as it provides good insight into the nature of geodesic elements.

\begin{lemma}\label{l:crit}
Let $\g$ be a Lie algebra and $G$ be its connected simply-connected Lie group. Suppose $\ip$ is an arbitrary inner product on $\g$. Then
\begin{enumerate}[{\rm (a)}]
  \item \label{it:crit1}
  a nonzero vector $Z \in \g$ is geodesic if and only if it is a critical point of the restriction of the squared norm function on $(\g, \ip)$ to the adjoint orbit of $Z$, 

  \item \label{it:crit2}
  if $\h$ is an ideal of $\g$ and $P$ is a Lie subgroup of $\mathrm{GL}(\h)$ containing $\Ad(G)_{|\h}$, then every closed nonzero orbit of the action of $P$ on $\h$ contains a geodesic element. 

\end{enumerate}
\end{lemma}
\begin{proof}
(a) A point $Z \ne 0$ is critical for the restriction of the squared norm function on $(\g, \ip)$ to the adjoint orbit of $Z$ if and only if for all $X \in \g$ we have
\begin{equation*}
0={\frac{d}{dt}}_{|t=0}\|\exp(t\ad(X))Z\|^2=2\<\ad(X)Z,Z\>,
\end{equation*}
which is equivalent to the fact that $Z$ is a geodesic element.

(b) Let $Y \in \h$ be a nonzero vector such that the orbit $P(Y) \subset \h$ is closed. Note that $0 \notin P(Y)$. Let $Z \in P(Y)$ be the closest point of $P(Y)$ from the origin. Then $Z$ is the closest point to the origin on the adjoint orbit $\Ad(G)Z$, hence is a geodesic element by assertion~\eqref{it:crit1}.
\end{proof}

\begin{remark}\label{rem:orbits}
In Lemma~\ref{l:crit}\eqref{it:crit2}, the subgroup $P$ may coincide with $\Ad(G)$ and the ideal $\h$ may coincide with $\g$. If $\g$ is a solvable Lie algebra with the nilradical $\n$, then the derived algebra $\g'$ is a subspace of $\n$, which implies that any nonzero vector orthogonal to $\n$ is geodesic. \end{remark}

We will require the following elementary proposition. We leave its proof to the reader; cf.~the proof of \cite[Prop.~1]{CLNN}.

\begin{proposition}\label{P1}
 If $\g $ has a nontrivial centre $\z$, then:
\begin{enumerate}[{\rm (a)}]
\item If $\g/\z$ possesses an inner product with a geodesic basis, then so too does $\g$.
\item If $\g/\z$ possesses an inner product with an orthonormal geodesic basis, then so too does $\g$. 
\end{enumerate}
\end{proposition}

We also require the following three results on orthonormal geodesic bases from~\cite{CLNN}.

\begin{proposition}\label{Tnil}
{\rm(\cite[Prop.~1]{CLNN}\rm)} If  $\g $ is a nilpotent Lie algebra, then every inner product on $\g$  has an orthonormal geodesic basis.
\end{proposition}

\begin{proposition}\label{P2}
{\rm(\cite[Prop.~2]{CLNN}\rm)}
 If $\g $ is unimodular and has a codimension one abelian ideal, then every inner product on $\g$  has an orthonormal geodesic basis.
\end{proposition}

\begin{theorem}\label{T0}
{\rm\cite[Theorem~1]{CLNN}\rm)} If  $\g $ is a unimodular Lie algebra of dimension $4$, then every inner product on $\g$  has an orthonormal geodesic basis.
\end{theorem}

\section{Abelian nilradical: Unimodular case}
\label{s:abelian}

Let $\g$ be a unimodular solvable Lie algebra with an abelian nilradical $\n$. Denote $n=\dim \n, \; m=\dim \g- n$. Choose an arbitrary basis $\{Y_j : \; j=1, \dots, m\}$, whose span complements $\n$ in $\g$ and denote $A_j, \; j= 1, \dots, m$, the restrictions of $\ad(Y_j)$ to $\n$. The operators $A_j$ pairwise commute and have zero trace.

For any choice of an inner product $\ip$ on $\g$, we can assume without loss of generality that the $Y_j$'s span the orthogonal complement to $\n$. Moreover, any nonzero element of $\n^\perp$ is a geodesic element. It follows that for a basis of geodesic elements for $\g$ to exist it suffices that such a basis exists for $\n$ (and the same is true for orthonormal bases). A nonzero vector $X \in \n$ is geodesic if and only if
\begin{equation}\label{eq:geodn}
    \<A_jX,X\>=0, \quad\text{for all } j=1, \dots, m.
\end{equation}

\begin{proof}[Proof of Theorem \ref{th:nilabel}]
In the above notation we have $m$ pairwise commuting operators $A_j$ on $\n$, with the zero trace. By \cite[Theorem~1]{Z} (see also \cite[Proposition~1.1]{AMS}), there exist a direct decomposition $\n=\oplus_{\a=1}^p \n_\a$, with the subspaces $\n_\a$ being common invariant subspaces of all the $A_j$, and bases $B^\a$ for each of the subspaces $\n_\a$ relative to which the restriction of every operator $A_j$ to every $\n_\a$ has one of the following forms:
\begin{equation*}
    \left(
      \begin{array}{ccc}
        c_j^\a &  & * \\
         & \ddots &  \\
        0 &  & c_j^\a \\
      \end{array}
    \right)
    \quad \text{or}\quad
    \left(
      \begin{array}{ccccc}
        a_j^\a & -b_j^\a &  &  &  \\
        b_j^\a & a_j^\a  &  & * &  \\
         &  & \ddots &  &  \\
         & 0 &  & a_j^\a  & -b_j^\a  \\
         &  &  & b_j^\a & a_j^\a  \\
      \end{array}
    \right)
\end{equation*}
(the matrix on the left is upper-triangular, with all the diagonal elements the same; the matrix on the right is upper block-triangular, with all the $2 \times 2$ diagonal blocks the same).

Choose the inner product on $\n$ in such a way that the subspaces $\n_\a$ are mutually orthogonal and that the elements of each basis $B^\a$ are orthonormal. Then for $e_\a \in B^\a, \, e_\b \in B^\b$, $\a \ne \b$, we have $\<A_je_\a, e_\b\>=0$, for all $j=1, \dots, m$, and moreover, $\<A_j e_\a, e_\a\> =n_\a^{-1} \Tr ((A_j)_{|\n_\a})$, where $n_\a=\dim \n_\a$. Now for every $\a$, take an arbitrary vector $e_\a^{i_\a} \in B^\a$, $i_\a=1, \dots, n_\a$, and an arbitrary number $\varepsilon_\a=\pm 1$ and define $Z:=\sum_\a \varepsilon_\a \sqrt{n_\a} e_\a^{i_\a}$. We have $\<A_j Z, Z\>=\sum_\a n_\a\<A_j e_\a^{i_\a}, e_\a^{i_\a}\>=\sum_\a n_\a(n_\a^{-1} \Tr ((A_j)_{|\n_\a}))=\Tr A_j=0$, by unimodularity, so every such vector $Z$ is geodesic. Note that if two such vectors $Z$ have all but one number $\varepsilon_\a$ the same, then their difference is a nonzero multiple of a basis vector. Hence such $Z$'s span the whole nilradical $\n$, as required.
\end{proof}

\section{Abelian derived algebra: $\mathbb{R}$-diagonal case}
\label{s:rdiag}

\begin{proof}[Proof of Theorem \ref{th:rdiag}]
Suppose $\g$ is a Lie algebra with an  abelian derived algebra $\g'$.
By Proposition \ref{P1}, we may assume that $\g$ has trivial centre. Let $\t$ be an arbitrary linear complement to $\g'$ in $\g$. As $\g$ is non-abelian, both $\t$ and $\g'$ are nontrivial. By hypothesis, the action of $\ad(\t)$ on $\g'$ is completely reducible over $\R$, so we have a root subspace decomposition $\g'= \n_1 \oplus \dots \oplus \n_p$, where $p \ge 1$ and the dimensions $d_\a=\dim \n_\a$ are all nonzero. For every root subspace $\n_\a$,  the corresponding root is the linear functional $\la_\a \in \t^*$ defined by $[Y,X]=\la_\a(Y)X$ for all $Y \in \t$ and any $X \in \n_\a$, where $\la_\a \ne \la_\b$ when $\a \ne \b$. Note that neither the root subspaces $\n_\a$, nor the roots depend on the particular choice of $\t$.

We first choose $\t$ in such a way that $[\t,\t]$ has the simplest form possible. To do this, for $\a=1, \dots, p$, choose a basis $\{X_{\a i}:  \; i=1, \dots, d_\a\}$ for $ \n_\a$, and define the two-forms $\omega_{\a i} \in \Lambda^2 \t^*$ as follows: for $Y_1, Y_2 \in \t$, set
\[
[Y_1, Y_2]=\sum_{\a=1}^p\sum_{i=1}^{d_\a} \omega_{\a i}(Y_1, Y_2)X_{\a i}.
\]
 Then the Jacobi identity on three elements of $\t$ implies that $\omega_{\a i} \wedge \lambda_\a=0$.
 So by Cartan's Lemma,  there exists a one-form $\psi_{\a i} \in \t^*$ such that $\omega_{\a i}= \psi_{\a i} \wedge \lambda_\a$. Now take an arbitrary basis $\{Y_a\}$ for $\t$ and define $\bar{Y}_a=Y_a-\sum_{\a=1}^p\sum_{i=1}^{d_\a} \psi_{\a i}(Y_a)X_{\a i}$. Then 
 \[
 [\bar{Y}_a, \bar{Y}_b]=\sum_{\a=1}^p\sum_{i=1}^{d_\a} (\omega_{\a i}(Y_a, Y_b)-\la_\a(Y_a)\psi_{\a i}(Y_b)+\la_\a(Y_b)\psi_{\a i}(Y_a)) X_{\a i} =0.
 \]
Now choosing $\bar{\t}$ to be the span of the $\bar{Y}_a$ and dropping all the bars, we obtain $[\t,\t] =0$.
  
  We will now construct the inner product $\ip$ on $\g$ having a basis of geodesic elements. Let $m = \dim \t$. As we will see, except for the case when $m =1$, the inner products on $\g'$ and on $\t$ can be chosen arbitrarily; it is the choice of $\g'^\perp$ (the ``inclination" of $\t$ to $\g'$) that really matters, and even it can be chosen from an open, dense set of linear complements to $\g'$ in $\g$. So when $m \not=1$, there is an abundance of inner products having a basis of geodesic elements. Note that, however, not every inner product has this property. For example, choosing $\t \perp \g'$ and the basis $\{X_{\a i}\}$ orthonormal, and assuming that for some $Y \in \t$, the restriction of $\ad(Y)$ to $\g'$ is positive definite, we obtain an inner product whose set of geodesic elements is $\t$. 

Suppose that $m > 1$. In the above notation, choose vectors $X'_a \in \g', \; a=1, \dots, m$, which will be specified later and define the inner product on $\g$ in such a way that $\g'^\perp=\Span_{a=1}^m(Y_a+X'_a)$ and that the subspaces $\n_\a$ are mutually orthogonal (this latter requirement is needed more for the sake of technical convenience).

For every $\a $, let $Y = \sum_a \mu_aY_a$ be a nonzero vector in $\Ker \la_\a$. Consider a vector $Z=\sum_a \mu_a(Y_a+X'_a) + \hat{X}$, where $\hat{X} \in \g'$ (note that the first summand of $Z$ lies in $\g'^\perp$). Such a vector is geodesic if and only if
\begin{equation}\label{eq:ZZT}
\<\hat{X},[Z,T]\>=0, \quad \text{for all } T \in \g.
\end{equation}
Now choosing $T=X_{\b i}$ in \eqref{eq:ZZT} we obtain $[Z,T]=[\sum_a \mu_a Y_a, X_{\b i}]=[Y, X_{\b i}]=\la_\b(Y)X_{\b i}$. It follows that $[Z,\n_\b] \subset \n_\b$, for $\b \ne \a$ and that $[Z,\n_\a]=0$, since $Y \in \Ker \la_\a$. Therefore $\ad(Z) \g' \subset \oplus_{\b \ne \a} \n_\b$, so for $T \in \g'$, condition \eqref{eq:ZZT} will be satisfied if $\hat X \perp \oplus_{\b \ne \a} \n_\b$, that is, when $\hat X \in \n_\a$. Next, choose $T \in \t$  in \eqref{eq:ZZT}. Then $\<\hat X,[Z,T]\>=\<\hat X,[\sum_a \mu_a X'_a + \hat{X}, T]\> + \<\hat X,[\sum_a \mu_a Y_a,T]\>$. But the latter term vanishes, since $[\t,\t] =0$, and the former term equals $\la_\a(T)\<\hat X, \pi_\a(\sum_a \mu_a X'_a) + \hat{X}\>$, where $\pi_\a: \g' \to \n_\a$ is the orthogonal projection. Thus the vector $Z$ is geodesic, for all choices of $\hat X \in \n_\a$ which satisfy the equation 
\[
0=\<\hat X, \pi_\a(\sum_a \mu_a X'_a) + \hat{X}\>=\big\|\hat X + \frac12 \pi_\a(\sum_a \mu_a X'_a)\big\|^2-\big\|\frac12 \pi_\a(\sum_a \mu_a X'_a)\big\|^2.
\]
 This equation for $\hat{X}$ defines a hypersphere in $\n_\a$, provided $\pi_\a(\sum_a \mu_a X'_a) \ne 0$. As any vector from $\g'^\perp$ is geodesic (in particular, $\sum_a \mu_a(Y_a+X'_a)$, the $\g'^\perp$-component of $Z$), we obtain that the span of the set of geodesic elements contains $\n_\a$, provided $\sum_a \mu_a \pi_\a(X'_a) \ne 0$. As not all of the $\mu_a$ are zeros (since $Y=\sum_a \mu_aY_a \ne 0$), we can choose (almost arbitrarily) the $\n_\a$ components $\pi_\a(X'_a)$ of the vectors $X'_a$ in such a way that this condition is satisfied. Repeating this procedure for every $\a$, we obtain that the span of the set of geodesic elements contains all the subspaces $\n_\a$, so it is the entire algebra $\g$, as required.

It remains to deal with the situation where $m = 1$, in which case $\g'$ is a codimension one abelian ideal. Here the required result follows from Theorem \ref{th:oneabel}, which we will prove in the next section.\end{proof}

\section{Codimension one ideals}\label{s:heis}

The fact that $\mathcal{A}_n=\R^{n}_{\id}$ does not have a geodesic basis has a stronger formulation which will be useful later in the paper.

\begin{lemma}\label{Lemma1}
For every inner product  on the Lie algebra $\mathcal{A}_n$,
every geodesic element is orthogonal to the derived algebra $\R^{n}$, so up to scaling, there is only geodesic element.
\end{lemma}
\begin{proof}
Choose an inner product $\ip$ on $\mathcal{A}_n$. Let $X$ be a nonzero element  orthogonal to $\mathcal{A}_n'$. Clearly $X$ is a geodesic element. Suppose that $X+Y$ is a geodesic element for some element $Y\in\mathcal{A}_n'$. We have
\[
0=\<[X,X+Y],X+Y\>=\<Y,X+Y\>=\<Y,Y\>,
\]
and so $Y=0$ as claimed.
\end{proof}

Given the above lemma, the proof of Theorem \ref{th:oneabel} can be completed 
directly by considering a Lie algebra of the form $\R^n_\varphi$, where $\varphi$ isn't the identity map, and using the rational canonical form of $\varphi$ to explicitly construct the required inner product. The following proof is more succinct.

\begin{proof}[Proof of Theorem \ref{th:oneabel}]
Suppose that $\g=\Span(Y) \oplus \h$ as a linear space where $\h$ is an abelian ideal, and denote $A=(\ad(Y))_{|\h} \in \End(\h)$. If $A=0$, then $\g$ is abelian and any nonzero vector of $\g$ is geodesic relative to any inner product. Otherwise, suppose that $A$ is not a multiple of the identity. Fix a background inner product $\ip_0$ on $\h$. Then for any positive definite, symmetric (relative to $\ip_0$) operator $G$ we can define an inner product $\ip$ on $\g$ by requiring that $\<Y, \h\>=0$ and that $\<X_1,X_2\>=\<GX_1,X_2\>_0$, for $X_1,X_2 \in \h$. For any choice of $G$, the vector $Y$ is geodesic for $(\g,\ip)$. Moreover, a nonzero vector $X \in \h$ is geodesic if and only if it is a zero of the quadratic form $\phi(X)=\<GAX,X\>_0$.

Now if $G$ is chosen in such a way that the form $\phi$ is indefinite, then its zero set spans $\h$ (and so the geodesic elements of $(\g,\ip)$ span the entire algebra $\g$). Indeed, suppose all the zeros of $\phi$ lie in a hyperplane of $\h$ defined by the equation $x_1=0$ relative to some basis. Then $\phi(X)>0$ when $x_1 > 0$ and $\phi(X) < 0$ when $x_1 < 0$ (or vice versa), so by continuity, $\phi(X)=0$ when $x_1=0$, so $\phi$ is a product of $x_1$ by a linear form, which must again be a multiple of $x_1$, as all the zeros of $\phi$ lie in the hyperplane $x_1=0$. So $\phi = c x_1^2$, a contradiction.

Therefore it is sufficient to find a symmetric positive definite $G$ such that $\phi$ is indefinite. Seeking a contradiction suppose that for any $G$ the form $\phi$ is semidefinite. If for some $G$ the form $\phi$ vanishes, then we are done. Otherwise, changing $A$ to $-A$ if necessary, we can assume that for some particular $G$, the form $\phi$ is positive semidefinite and nonzero. By continuity and from the fact that the set of positive definite, symmetric operators $G$ is connected, it follows that this is true for all $G$, that is, $\<GAX,X\>_0 \ge 0$, for all $X \in \h$ and all positive definite, symmetric operators $G$. As $A$ is not a multiple of the identity, there exists $X' \in \h$ such that $\rk(X',AX')=2$. Define $T=\|AX'\|_0 X'-\|X'\|_0 AX'$ and then define $G$ by $GX=\<T,X\>_0 T+\ve X$ for $X \in \h$, where $\ve > 0$ is small enough. Such a $G$ is symmetric and positive definite and moreover, $\<GAX',X'\>_0=\<T,AX'\>_0\<T,X'\>_0 + \ve \<AX',X'\>_0=-\|AX'\|_0 \|X'\|_0 (\<X', AX'\> -\|X'\|_0 \|AX'\|_0)^2 + \ve \<AX',X'\> < 0$ for small enough $\ve$, as $AX'$ and $X'$ are non-collinear. This contradiction proves the proposition.
\end{proof}

We now consider a Lie algebra $\g$ having a codimension one ideal isomorphic to the Heisenberg Lie algebra $\H_{2m+1}$. 
Recall that $\H_{2m+1}$ has basis $\{X_1,\dots,X_{2m+1}\}$ and relations $[X_i,X_{i+m}]=X_{2m+1}$, for $ i=1, \dots, m$.

\begin{proof}[Proof of Theorem \ref{th:oneheis}]
Let $\ip$ be an inner product on $\g$. Let $\m$ be the orthogonal complement to $X_{2m+1}$ in $\H_{2m+1}$ and let $Y$ be a unit vector orthogonal to $\H_{2m+1}$. Note that $X_{2m+1}$ spans the centre of $\H_{2m+1}$, which is a characteristic ideal and hence $\ad(Y)$-invariant. So $[Y,X_{2m+1}]=\la X_{2m+1}$ for some $\la \in \R$. The Lie bracket on $\m$ is given by $[W_1,W_2]=\<KW_1,W_2\>X_{2m+1}$ for all $W_1, W_2 \in \m$, where $K$ is a nonsingular skew-symmetric operator on $\m$. Furthermore, there exist $A \in \End(\m)$ and a $1$-form $\omega$ on $\m$ such that for $W \in \m$ we have $[Y,W]=AW+\omega(W)X_{2m+1}$. The fact that $\ad(Y)$ is a derivation on $\H_{2m+1}$ gives $\<KAW_1,W_2\>+\<KW_1,AW_2\>=\la\<KW_1,W_2\>$, for all $W_1,W_2 \in \m$, that is, $KA+A^tK=\la K$. Then $A=\la \id -K^{-1}A^tK$. Thus $\Tr A= 2m\la -\Tr A$,
so $\Tr A=m\la$. Hence $\Tr \ad(Y)=\Tr A+ \la = (m+1)\la$. Note that  $\Tr \ad(X)=0$, for all $X \in \H_{2m+1}$. Hence $\g$ is unimodular if and only if $\la = 0$.

Suppose $\g$ is nonunimodular, so that $\la \ne 0$. Then for any nonzero $T \in \g$, we have $X_{2m+1} \in \Img \ad(T)$. Indeed, let $T=aY+W+bX_{2m+1} \ne 0$, with $W \in \m$. If $a \ne 0$, then we have $[T,(a \la)^{-1}X_{2m+1}]=X_{2m+1}$. If $a = 0$, but $W \ne 0$, then $[T,KW]=\|KW\|^2X_{2m+1}$. If $T=bX_{2m+1}$ with $ b \ne 0$, then $[T,Y]=-b \la X_{2m+1}$. It now follows from \eqref{eq:def} that every geodesic element of $(\g,\ip)$ is orthogonal to $X_{2m+1}$. In particular, the geodesic elements do not span the entire algebra $\g$.

Now suppose $\g$ is unimodular. Then $\la = 0$ and $X_{2m+1}$ belongs to the centre of $\g$. The quotient algebra $\g/ \Span(X_{2m+1})$ is unimodular and has a codimension one abelian ideal. So by Proposition \ref{P2}, every inner product on $\g/ \Span(X_{2m+1})$ has an orthonormal geodesic basis. Hence by Proposition \ref{P1}, every inner product on $\g$ has an orthonormal geodesic basis. \end{proof}

\section{Dimension $\le 4$}\label{s:dim4}

This section is devoted to the proof of Theorem \ref{T:dim4}. 
The classification of real Lie algebras of dimension $\leq 4$ is quite old and has been verified by several authors. Apart from the simple algebras $\sl(2,\R)$ and $\so(3,\R)$, and the corresponding reductive algebras $\sl(2,\R)\oplus \R$ and $\so(3,\R)\oplus \R$, the algebras are all solvable. We will use the classification of solvable Lie algebras of dimension $\leq 4$ by de Graaf   \cite{Graaf}. 
By Theorem \ref{T0}, we need only consider nonunimodular Lie algebras. The nonunimodular Lie algebras over $\R$ of dimension two and three  have a codimension one abelian ideal, so that are handled by Theorem \ref{th:oneabel}. For nonunimodular Lie algebras of dimension 4 over $\R$, one has the following complete list, in the notation of \cite{Graaf}:

\bigskip

\begin{enumerate}\itemsep2pt

\item[$M^2$:]\label{M2} $[X_1,X_2]=X_2$, $[X_1,X_3]=X_3$, $[X_1,X_4]=X_4$.

\item[$M^3_a$:]\label{M3} $[X_1,X_2]=X_2$, $[X_1,X_3]=X_4$, $[X_1,X_4]=-aX_3+(a+1)X_4$.

\item[$M^4$:]\label{M4} $[X_1,X_2]=X_3$, $[X_1,X_3]=X_3$.

\item[$M^6_{a,b}$:]\label{M6} $[X_1,X_2]=X_3$, $[X_1,X_3]=X_4$, $[X_1,X_4]=aX_2+bX_3+X_4$.

\item[$M^8$:]\label{M8} $[X_1,X_2]=X_2$, $[X_3,X_4]=X_4$.

\item[$M^9_{-1}$:]\label{M9} $[X_1,X_2]=X_2-X_3$, $[X_1,X_3]=X_2$, $[X_4,X_2]=X_2$, $[X_4,X_3]=X_3$.

\item[$M^{12}$:]\label{M12} $[X_1,X_2]=X_2$, $[X_1,X_3]=2X_3$, $[X_1,X_4]=X_4$, $[X_4,X_2]=X_3$.

\item[$M^{13}_a$:]\label{M13} $[X_1,X_2]=X_2+aX_4$, $[X_1,X_3]=X_3$, $[X_1,X_4]=X_2$, $[X_4,X_2]=X_3$.

\end{enumerate}

\bigskip

First note that for the algebras $M^2,M_a^3,M^4$ and $M^6_{a,b}$, the ideal $\Span(X_2,X_3,X_4)$ is abelian in each case.  
So each of the algebras $M^2,M_a^3,M^4$ and $M^6_{a,b}$ is of the form $\R^3_{\va}$, where $\va=\ad(X_1)$.
Hence by Theorem \ref{th:oneabel},  $M^2$ doesn't have an inner product with a geodesic basis, while  $M_a^3,M^4$ and $M^6_{a,b}$ do have such an inner product.

For $M^8$, the derived algebra is the abelian algebra $\Span(X_2,X_4)$, and by Theorem  \ref{th:rdiag},  $M^8$ has an inner product with a geodesic basis. In fact, an explicit example is easily furnished;  taking the
elements $X_1+X_4,X_2,X_2+X_3,X_4$ to be orthonormal, the elements $X_1,
X_1+X_4,X_3,X_2+X_3$ form a geodesic basis.

For $M^9_{-1}$, the unimodular kernel is $\Span(X_4-2X_1,X_2,X_3)$, which is isomorphic to the Lie algebra of the group of isometries of the Euclidean plane.
Observe that by Lemma \ref{Lemma1}, for any given inner product, every geodesic element in the ideal $\h=\Span(X_2,X_3,X_4)$ lies in the orthogonal complement of $\Span(X_2,X_3)$. If for some constants $a,b,c$, the   element $U=X_1+aX_2+bX_3+cX_4$ is geodesic, we would have
\begin{align*}
\<[X_3,U],U\>& =0\qquad \implies \qquad \<X_2+cX_3,U\>=0,\\
\<[X_2,U],U\>& =0\qquad \implies \qquad \<-(c+1)X_2+X_3,U\>=0,
\end{align*}
and so $\<X_2,U\>=-c\<X_3,U\>$ and $(c^2+c+1)\<X_3,U\>=0$.
Since $c^2+c+1$ has no roots in $\R$, we conclude that
$U$ is orthogonal to both $X_2$ and $X_3$. Thus all geodesics in $\g$ are orthogonal to $\Span(X_2,X_3)$. In particular, $M^9_{-1}$ doesn't have an inner product with a geodesic basis.

For the algebras $M^{12}$ and $M^{13}_a$, the unimodular kernel is $\Span(X_2,X_3,X_4)$, which is isomorphic to $\H_3$. So $\g$ doesn't have an inner product with a geodesic basis, by Theorem \ref{th:oneheis}(a).

\section{Dimension 5: Initial remarks}\label{s:dim4}

The remainder of the paper is devoted to the proof of Theorem \ref{T:dim5}.
Let $\g$ be a unimodular Lie algebra of dimension 5.

\begin{remark}\label{R5b}
 If the centre $\z$ of $\g$ is nontrivial, then the Lie algebra $\g/\z$ is unimodular and of dimension at most 4. So $\g/\z$ has an inner product with an orthonormal geodesic basis. Then by Proposition \ref{P1},  $\g$ has an inner product with an orthonormal geodesic basis. So we may assume that $\g$ has trivial centre.
\end{remark}

\begin{remark}\label{R5a}
 If $\g$ is a direct sum of proper ideals, then these ideals are unimodular and one of them, $\h$ say, must have dimension $\leq2$. So $\h$ is abelian, and hence $\h$ belongs to the centre of $\g$. So by the previous remark, we may assume that $\g$ is indecomposable.
\end{remark}

Given the above two remarks, for the remainder of the paper, $\g$ is an indecomposeable unimodular Lie algebra of dimension 5 with trivial centre.

\section{Dimension 5:  Nonsolvable algebras}
Suppose that $\g$ is nonsolvable.
The only semisimple Lie algebras of dimension $\leq 5$ are the 3-dimensional simple algebras $\so(3,\R)$ and $\sl(2,\R)$. So $\g$ has a two-dimensional radical, and hence as $\g$ is unimodular, its radical is isomorphic to $\R^2$. Let $\ve$ denote a Levi subalgebra of $\g$, so $\g$ is isomorphic to a semidirect product $\ve\ltimes \R^2$ determined by a homomorphism $\ve\to \sl(2,\R)$. It is easy to see that since $\so(3,\R)$ and $\sl(2,\R)$ are simple, there is no nontrivial Lie algebra homomorphism from $\so(3,\R)$ to $\sl(2,\R)$, and the only nontrivial Lie algebra homomorphisms from $\sl(2,\R)$ to $\sl(2,\R)$ are Lie algebra isomorphisms. Hence, up to isomorphism, there are only three possibilities for $\g$:
\[
\so(3,\R)\oplus \R^2,\qquad \sl(2,\R)\oplus \R^2,\qquad\sl(2,\R)\ltimes  \R^2,
\]
where in the third case, $\sl(2,\R)$ acts on $\R^2$ by its canonical linear action.
As $\g$ is  indecomposable, it remains to consider the third case.
Consider the basis $\{X_1,\dots, X_5\}$  for $\sl(2,\R)\ltimes \R^2$ defined as follows: 
\begin{align*}
X_1&=
\begin{pmatrix}
1&0\\
0&-1
\end{pmatrix},\quad
X_2=
\begin{pmatrix}
0&1\\
0&0
\end{pmatrix},\quad
X_3=
\begin{pmatrix}
0&0\\
1&0
\end{pmatrix},\quad
X_4=
\begin{pmatrix}
1\\
0
\end{pmatrix},\quad
X_5=
\begin{pmatrix}
0\\
1
\end{pmatrix}.
\end{align*}
The nonzero relations are:
\begin{align*}
[X_1,X_2]&=2X_2,\qquad[X_1,X_3]=-2X_3,\qquad [X_1,X_4]=X_4,\qquad[X_1,X_5]=-X_5,\\
[X_2,X_3]&=X_1,\qquad[X_3,X_4]=X_5,\qquad[X_2,X_5]=X_4.
\end{align*}

\begin{proposition}\label{P} The $5$-dimensional unimodular Lie algebra $\sl(2,\R)\ltimes \R^2$ possesses an inner product with a geodesic basis. 
\end{proposition}

\begin{proof} 
 Consider the metric $\ip$ on $\g=\sl(2,\R)\ltimes \R^2$  for which
 $X_1,\dots,X_5$  are orthonormal. Clearly $X_1$ is geodesic.
Consider the elements 
\[Y_1=X_2+{\sqrt2}X_5,\quad
  Y_2=X_2-{\sqrt2}X_5\quad
 Y_3=X_3+{\sqrt2}X_4,\quad
  Y_4=X_3-{\sqrt2}X_4.
\]
Clearly $X_1,Y_1,\dots,Y_4$ span $\g$. 
We claim that the elements $Y_1,\dots,Y_4$ are geodesic. 
We have
\begin{enumerate}
\item $[X_1,Y_1]=2X_2-{\sqrt2}X_5$, 
\item $[X_2,Y_1]={\sqrt2}X_4$, 
\item $[X_3,Y_1]=-X_1$, 
\item $[X_4,Y_1]=0$, 
\item $[X_5,Y_1]=-X_4$, 
\end{enumerate}
These elements are all orthogonal to $Y_1$, so $Y_1$ is geodesic. Similar calculations show that $Y_2,Y_3,Y_4$ are geodesic.\end{proof}

We now consider orthonormal bases.

\begin{proposition}\label{p:noort}
The Lie algebra $\g=\sl(2,\R) \ltimes \R^2$ admits no inner product with an orthonormal basis of geodesic elements.
\end{proposition}

\begin{proof}
The algebra $\g$ is isomorphic to (and will be, throughout the proof, identified with) the Lie algebra of $3 \times 3$ matrices of zero trace having zero third row. We will denote the elements of $\g$ by $X=\big(\begin{smallmatrix} A & x \\ 0 & 0 \end{smallmatrix}\big)$, where $x \in \R^2$ and $A$ is a $2 \times 2$ matrix with $\Tr  A =0$. For $X_i=\big(\begin{smallmatrix} A_i & x_i \\ 0 & 0 \end{smallmatrix}\big) \in \g, \; i=1,2$, the Lie bracket and the Killing form $b$ on $\g$ are given by
\begin{equation}\label{eq:sl2bracket}
    [X_1, X_2]=\begin{pmatrix} [A_1,A_2] & A_1x_2-A_2x_1 \\ 0 & 0 \end{pmatrix} \, , \quad  b(X_1, X_2)=\Tr   (A_1A_2),
\end{equation}
(up to a multiple). Note that $b$ is degenerate. We call $X=\big(\begin{smallmatrix} A & x \\ 0 & 0 \end{smallmatrix}\big) \in \g$ \emph{singular} if $\det A = 0$, and \emph{nonsingular} otherwise (this is well-defined, as $\det A= -\frac12 b(X,X)$). The elements of $\g$ having $A=0$ form an abelian ideal $\ag= \R^2$.

It is easy to see that the conjugations $\phi_T:X \mapsto TXT^{-1}$, where $T=\big(\begin{smallmatrix} M & u \\ 0 & 1 \end{smallmatrix}\big)$, $M$ is a nonsingular $2 \times 2$ matrix and $u \in \R^2$, are automorphisms of $\g$. We have
\begin{equation}\label{eq:conjug}
    \phi_TX=\begin{pmatrix} MAM^{-1} & -MAM^{-1}u+Mx \\ 0 & 0 \end{pmatrix}\, .
\end{equation}

{\begin{lemma}\label{l:nork1}
No $X \in \g$ with $\rk X =1$ \emph{(}in particular, no $X \in \ag$\emph{)} is a geodesic element. 
\end{lemma}
\begin{proof} Any such $X$ belongs to $\Img \ad(X)$, which can be easily seen by reducing $X$ by a conjugation $\phi_T$ to a form with only the first row being nonzero. 
\end{proof}
}
From \eqref{eq:conjug} we can see that any $X \in \g$ with $\rk X = 2$ can be reduced by a scaling and a conjugation $\phi_T$ to one of the following canonical forms:
\begin{equation}\label{eq:canon}
    C_1=\begin{pmatrix} 0 & 1 & 0 \\ 0 & 0 & 1 \\ 0 & 0 & 0 \end{pmatrix}\, ,  \quad    C_2=\begin{pmatrix} 1 & 0 & 0 \\ 0 & -1 & 0 \\ 0 & 0 & 0 \end{pmatrix}\, , \quad    C_3=\begin{pmatrix} 0 & 1 & 0 \\ -1 & 0 & 0 \\ 0 & 0 & 0 \end{pmatrix}\, ,
\end{equation}
depending on whether $X$ is singular or nonsingular and whether, in the latter case, the eigenvalues of $A$ are real or imaginary, respectively. 

Suppose now that for some inner product $\ip$ on $\g$ there exists an orthogonal basis $\mathcal{B}$ of geodesic elements.

{\begin{lemma}\label{l:nononsing}
The basis $\mathcal{B}$ contains no nonsingular vectors. 
\end{lemma}
\begin{proof} 
Suppose that $X$ is a nonsingular geodesic element. It follows from \eqref{eq:sl2bracket} (and can be easily verified using \eqref{eq:canon}), that as $X$ is nonsingular,  $\dim \Img \ad(X) = 4$.  Note that $\Img \ad(X)$ lies in the orthogonal complement to $X$ relative to the Killing form $b$; indeed, if $X=\big(\begin{smallmatrix} A & x \\ 0 & 0 \end{smallmatrix}\big)$ and $Y=\big(\begin{smallmatrix} B & y \\ 0 & 0 \end{smallmatrix}\big)$, then $b(X,[X,Y])=\Tr(A\,[A,B])=\Tr(AAB)-\Tr(ABA)=0$. Thus, since  $\Img \ad(X)$ has dimension 4, $\Img \ad(X)$ equals the orthogonal complement to $X$ relative to the Killing form $b$.  Note that since $X$ is geodesic, $\Img \ad(X)$ lies in the orthogonal complement to $X$ relative to $\ip$. So, as  $\Img \ad(X)$ has dimension 4, $\Img \ad(X)$ also coincides with the orthogonal complement to $X$ relative to $\ip$.
In particular, as $\ag\subset \Img \ad(X)$, we have that $X$ is perpendicular to $ \ag$ relative to $\ip$. It follows that $\mathcal{B}$ cannot contain more than two nonsingular vectors. Indeed, if it contained three such vectors, then the other two would be in  $\ag$, thus contradicting Lemma~\ref{l:nork1}.

If $\mathcal{B}$ contains exactly two nonsingular vectors $X_i=\big(\begin{smallmatrix} A_i & x_i \\ 0 & 0 \end{smallmatrix}\big), \; i=1,2$, then $X_1,X_2$ are $\ip$-orthogonal, and hence from what we have just seen,  $X_1,X_2$ are $b$-orthogonal.
So $b(X_1,X_2)=\Tr   (A_1A_2) = 0$. Similarly,  for any other basis element $X=\big(\begin{smallmatrix} A & x \\ 0 & 0 \end{smallmatrix}\big) \in \mathcal{B}$, we have $\Tr  (A_1 A)= \Tr  (A_2 A)=0$. So $A$ is either zero or nonsingular (this follows from the nondegeneracy of the Killing form on $\sl(2,\R)$ or can be verified directly from the canonical forms \eqref{eq:canon}). The latter case is impossible by assumption, the former, by Lemma~\ref{l:nork1}.

It remains to suppose $\mathcal{B}$ contains exactly one nonsingular vector $X_1=\big(\begin{smallmatrix} A_1 & x_1 \\ 0 & 0 \end{smallmatrix}\big)$. Then, arguing as above, for any other basis vector $X =\big(\begin{smallmatrix} A & x \\ 0 & 0 \end{smallmatrix}\big) \in \mathcal{B}$ we have $\Tr   (A_1A)=0$ and $\det A=0$. Without loss of generality we can assume that $X_1$ has one of the canonical forms $C_2$ or $C_3$ from \eqref{eq:canon}. If $X_1=C_3$, then $A=0$, which contradicts Lemma~\ref{l:nork1}. If $X_1=C_2$, then up to scaling we have
\begin{equation} \label{eq:2sing}
X=\begin{pmatrix}
0 & 1 & c_1  \\
0 & 0 & c_2 \\
0 & 0 & 0 
\end{pmatrix} \text{ or }
X=\begin{pmatrix}
0 & 0 & d_1  \\
1 & 0 & d_2 \\
0 & 0 & 0
\end{pmatrix}.
\end{equation}
If $c_1 \ne 0$ in the first case of \eqref{eq:2sing}, then taking $Y=\Big(\begin{smallmatrix} 1 & 0 & 0 \\ 3c_2c_1^{-1} & -1 & -c_1 \\ 0&0&0\end{smallmatrix}\Big)$ we obtain $2X=[Y,X]+3c_2c_1^{-1}X_1$, which is a contradiction, as  $X $ is perpendicular to both $X_1$ and $\Img\ad(X)$. Similarly, if $d_2 \ne 0$ in the second case of \eqref{eq:2sing}, then for $Y=\Big(\begin{smallmatrix} 1 & -3 d_1 d_2^{-1} & d_2 \\ 0 &-1&0\\ 0&0&0 \end{smallmatrix}\Big)$ we get $[X,Y]-3d_1d_2^{-1} X_1=2X$, which is also a contradiction. It follows that the remaining four vectors of $\mathcal{B}$ up to scaling have the form
\begin{equation*} 
X_i=\begin{pmatrix}
0 & 1 & 0  \\
0 & 0 & \mu_i \\
0 & 0 & 0
\end{pmatrix}, \; i=2,3, \text{ and }
X_j=\begin{pmatrix}
0 & 0 & \nu_j  \\
1 & 0 & 0 \\
0 & 0 & 0
\end{pmatrix}, \; j=4,5.
\end{equation*}
But then $[C_2,X_2]=\Big(\begin{smallmatrix} 0 & 2 & 0 \\ 0 &0&-\mu_2\\ 0&0&0 \end{smallmatrix}\Big) \in \Span(X_2,X_3)$. As $X_2$ is perpendicular to both $X_3$ and $ [C_2,X_2]$, the vector $[C_2,X_2]$ must be a multiple of $X_3$, so $-\mu_2=2\mu_3$. A similar argument shows that $-\mu_3=2\mu_2$, so $\mu_2=\mu_3=0$, which is a contradiction. 
\end{proof}
}
It follows that all the elements $X_i \in \mathcal{B}$ are singular. Without loss of generality we can assume that $X_1=C_1$. Furthermore, at least one of the $X_i, \; i > 1$, has a nonzero $(1,1)$-entry. As it is singular and has zero trace, both its $(2,1)$ and $(1,2)$ entries must be nonzero. So up to scaling and relabelling we can assume that $X_2=\Big(\begin{smallmatrix} t & 1 & x \\ -t^2 &-t&y\\ 0&0&0 \end{smallmatrix}\Big)$, where $t \ne 0$. Then the conjugation $\phi_T$ with $T=\Big(\begin{smallmatrix} 1 & t^{-1} & -yt^{-2} \\ 0 &1&t^{-1}\\ 0&0&1 \end{smallmatrix}\Big)$ stabilises $X_1$ and $\phi_TX_2=\Big(\begin{smallmatrix} 0 & 0 & x+yt^{-1} \\ -t^2 &0&0\\ 0&0&0 \end{smallmatrix}\Big)$, so without loss of generality we can take $X_2=\Big(\begin{smallmatrix} 0 & 0 & a \\ 1 &0&0\\ 0&0&0 \end{smallmatrix}\Big)$, where $a \ne 0$ (by Lemma~\ref{l:nork1}). Furthermore, acting by $\phi_T$, with $T = \mathrm{diag}(a^{-2/3},a^{-1/3},1)$, and scaling by a factor of $a^{1/3}$, we get
\begin{equation*} 
X_1=\begin{pmatrix}
0 & 1 & 0  \\
0 & 0 & 1 \\
0 & 0 & 0
\end{pmatrix}, \qquad
X_2=\begin{pmatrix}
0 & 0 & 1  \\
1 & 0 & 0 \\
0 & 0 & 0
\end{pmatrix}.
\end{equation*}
Now for an arbitrary $X=\Big(\begin{smallmatrix} \alpha & \beta & \mu \\ \gamma &-\alpha & \nu\\ 0&0&0 \end{smallmatrix}\Big) \in \g$ we have
\begin{equation}\label{eq:imX1X2}
[X_1,X]=\begin{pmatrix}
\gamma & -2\alpha & \nu-\beta  \\
0 & -\gamma & \alpha \\
0 & 0 & 0
\end{pmatrix}, \qquad
[X_2,X]=\begin{pmatrix}
-\beta & 0 & -\alpha \\
2 \alpha & \beta & \mu-\gamma\\
0 & 0 & 0
\end{pmatrix}.
\end{equation}
Denote $\ve_{ij}, \; i=1,2,\, j=1,2,3$, the linear functionals on $\g$ such that $\ve_{ij}(X)$ is the $(i,j)$-th entry of $X$. As the basis $\mathcal{B}$ is orthogonal and geodesic, we have $X_1 \perp \Img \ad(X_1) \oplus \Span(X_2)$, so $\Span(X_2,X_3,X_4,X_5)=\Img \ad(X_1)\oplus \Span(X_2)$. Then from \eqref{eq:imX1X2} it follows that $\Span(X_2,X_3,X_4,X_5)=\Ker(\ve_{12}+2\ve_{23})$. Similar arguments applied to $X_2$ give $\Span(X_1,X_3,X_4,X_5)=\Ker(\ve_{21}+2\ve_{13})$. Therefore $\Span(X_3,X_4,X_5)=\Ker(\ve_{12}+2\ve_{23}) \cap \Ker(\ve_{21}+2\ve_{13})$. As all the $X_i$ are singular, we obtain, up to scaling,
\begin{equation}\label{eq:Xuv}
X_i=\begin{pmatrix}
2u_iv_i & 2u_i^2 & v_i^2  \\
-2v_i^2 & -2u_iv_i & -u_i^2 \\
0 & 0 & 0
\end{pmatrix}, \; i=3,4,5,
\end{equation}
where $(u_i,v_i) \ne (0,0)$. Let us just consider $X_3$ and for convenience we will drop the subscripts from $u_3$ and $v_3$. We have $X_3 \perp \Img\ad(X_3)+\Span(X_1,X_2)$. If 
\[
\Img\ad(X_3)+\Span(X_1,X_2) \supset \ag,
\] then for reasons of dimension, $\Img\ad(X_3)+\Span(X_1,X_2)=\ag \oplus\Span(X_1,X_2) =\Ker(\ve_{11})$, which is not possible, since for $Y=\Big(\begin{smallmatrix} 0 & 1 & 0 \\ 1 & 0 & 0 \\ 0&0&0\end{smallmatrix}\Big)$ we have $[X_3,Y]\in \Img\ad(X_3)$ but $\ve_{11}([X_3,Y])=2 (u^2+v^2) \ne 0$. Therefore
\[
\dim ((\Img\ad(X_3) +\Span(X_1,X_2))\cap \ag) \le 1.\]
Then for $Y=\Big(\begin{smallmatrix} 0 & 0 & v \\ 0 & 0 & u\\ 0&0&0\end{smallmatrix}\Big)$, the vectors
\begin{align*}
&[X_3,C_2]+4u^2X_1+4v^2X_2=\begin{pmatrix}
0 & 0 & 3v^2  \\
0 & 0 & 3u^2 \\
0 & 0 & 0
\end{pmatrix} \in (\Img\ad(X_3) +\Span(X_1,X_2))\cap \ag,\\
&[X_3,Y]=2 (u^2+v^2) \begin{pmatrix}
0 & 0 & u  \\
0 & 0 & -v \\
0 & 0 & 0
\end{pmatrix} \in (\Img\ad(X_3) +\Span(X_1,X_2))\cap \ag
\end{align*}
must be collinear, which implies $u=-v$. So from \eqref{eq:Xuv},
\begin{equation*}
X_3=u^2\begin{pmatrix}
-2 & 2 & 1  \\
-2 & 2 & -1 \\
0 & 0 & 0
\end{pmatrix}.
\end{equation*}
By the same reasoning, $X_4$ and $X_5$ have the same form. 
Hence $X_3,X_4,X_5$ are not linearly independent, which is a contradiction.
\end{proof}

\section{Dimension 5:  Solvable algebras}
The classification of 5-dimensional real Lie algebras due to Mubarakzjanov \cite{Mub1,Mub2} is possibly not as well known as the classification in dimension 4, but as far as we are aware, it is error free. We will only require it in certain cases; in many cases, we will give general arguments that don't rely on  this classification. 

Mubarakzjanov's classification is presented according to the nilradical $\n$ of $\g$. We follow the same presentation, and use the notation of \cite{Mub1,Mub2} where appropriate.
As $\g$ has dimension 5, we have $\dim(\n)\in\{3,4,5\}$ (see \cite[Theorem 5]{Mub1}). If $\dim(\n)=5$, then $\g$ is nilpotent and an inner product with an orthonormal  geodesic basis exists by Proposition \ref{Tnil}. So we are left with the cases $\dim(\n)=3$ and $\dim(\n)=4$. Up to isomorphism, there are only two nilpotent Lie algebras of dimension three ($\R^3$ and $\H_3$) and three nilpotent Lie algebras of dimension 4 ($\R^4, \H_3\oplus\R$ and $\m_0(4)$), where  $\m_0(4)$ is the filiform Lie algebra
having basis $\{X_1,\dots,X_4\}$ and relations $[X_1,X_2]=X_3,\ [X_1,X_3]=X_4$. 
So we have 5 subcases to consider.

\subsection{$\n\cong\R^4$.}
Here, since $\g$ is unimodular,  an inner product with an orthonormal basis exists by Proposition \ref{P2}.

\subsection{$\n\cong \H_3\oplus\R$.}
Choose a basis $\{X_0,X_1,\dots, X_4\}$  for $\g$ so that $\n=\Span(X_1,\dots, X_4)$, $\H_3=\Span(X_1,X_2,X_3)$ where $[X_1,X_2]=X_3$ and $\R=\Span(X_4)$. Let $\varphi=\ad(X_0)|_{\n}$. Since the centre $V:=\Span(X_3,X_4)$  of $\n$ and the derived algebra $\Span(X_3)$ of $\n$ are characteristic ideals of $\n$, they are invariant under the derivation $\varphi$ (see \cite[Chap.~1.1.3]{Bou}). Thus the matrix representation $A$ of $\varphi$ relative to the basis $\{X_1,\dots, X_4\}$ has the form
\begin{equation} \label{mrep1}
A=\begin{pmatrix}
a & b & 0 & 0 \\
c & d & 0 & 0 \\
\alpha  & \beta  &   \lambda   &   e     \\
\gamma & \delta      & 0    &f
\end{pmatrix}.
\end{equation}
Notice that as $\varphi$ is a derivation, applying $\varphi$ to $[X_1,X_2]$ gives
\begin{equation} \label{deri}
a+d=\lambda.
\end{equation}
Then, as $\varphi$ has zero trace, $f=-2\lambda$. Notice that $\lambda\not=0$ since otherwise $X_3$ is in the  centre of $\g$, contrary to the assumption that $\g$ has trivial centre.  

Suppose that $\g$ possesses an inner product for which there is an orthonormal geodesic basis 
$\{Y_0,Y_1,\dots, Y_4\}$. Consider an element $Y_i\not\in V$. If $Y_i\not\in \Span(X_1,\dots,X_4)$, then $[Y_i,X_3]$ is a nonzero multiple of $X_3$. If $Y_i\in \Span(X_1,\dots,X_4)$, then either $ [X_1,Y_i]$ or $[X_2,Y_i]$ is a nonzero multiple of $X_3$. 
So in all cases, as $Y_i$ is geodesic,  $Y_i$ is orthogonal to $X_3$. Thus there are at most four numbers $i$ in $\{0,1,\dots,4\}$ with $Y_i\not\in V$. If there are four of them, the corresponding $Y_i$ are all orthogonal to $X_3$; then since the elements $Y_0,\dots, Y_4$ are orthonormal, one of these is a multiple of $X_3$. But then $X_3$ would be geodesic, which is false as $\lambda\not=0$. So we may assume that there are precisely three numbers $i$ in $\{0,1,\dots,4\}$ with $Y_i\not\in V$; say $Y_0,Y_1,Y_2$. Hence $Y_3,Y_4\in V$, and so $V=\Span(Y_3,Y_4)$.
As we mentioned above, $V$ is $\varphi$-invariant. Thus, as $Y_3$ is geodesic, $\varphi(Y_3)$ is a multiple of $Y_4$, and similarly, $\varphi(Y_4)$ is a multiple of $Y_3$. In particular, $ \varphi|_{V}$ has zero trace. But this impossible as $ \Tr  \varphi|_{V}=\lambda+f=-\lambda\not=0$. So  $\g$ does not have an inner product with an orthonormal basis.

To see that in this case there is nevertheless an inner product with a (nonorthonormal) geodesic basis, we will employ \cite{Mub1,Mub2}. According to this classification, for unimodular Lie algebras with trivial centre, when $\n\cong \H_3\oplus\R$, the matrix $A$ of \eqref{mrep1} may be taken to be  one the following 4 forms:

    \begin{enumerate}
    \item $\g_{19}(\alpha)$ with $\alpha\not= -1$:
    \[
    A=\begin{pmatrix} 1 & 0 & 0 & 0 \\ 0 & \alpha & 0 &0  \\ 0 &  0 & 1+\alpha & 0 \\ 0 & 0 & 0 & -2(1+\alpha) \end{pmatrix}.\]

\item $\g_{23}$:
\[
A=\begin{pmatrix} 1 & 0 &0 &0  \\ 1 & 1 &0 & 0\\   0 & 0 & 2 & 0 \\  0 & 0 & 0 & -4 \end{pmatrix}.
\]

\item $\g_{25}(p)$ with $p\not=0$:
\[
A=\begin{pmatrix}  p & -1 &0 &0  \\  1 & p &0 & 0 \\ 0 & 0 &2p &  0 \\ 0 & 0 & 0 & -4p \end{pmatrix}.\]

\item $\g_{28}(-\frac 32)$:
\[A=\begin{pmatrix} -\frac 32 & 0 & 0 & 0 \\ 0 & 1 & 0 &0  \\ 0 & 0 & -\frac 12 & 0 \\ 0 & 1 & 0 & 1 \end{pmatrix}.
\]
 \end{enumerate}
For each of the above 4 cases, we take the inner product on $\g$ for which the elements $X_0,X_1,\dots,X_4$ are orthonormal.   
Then for $\g_{23}$ and $\g_{25}(p)$ it is easy to verify that the following elements form a geodesic basis:
\[
X_0,\ 2X_1+X_4,\
2X_2+X_4,\
\pm \sqrt{2} X_3+X_4,
\]
For $\g_{28}(-\frac 32)$, 
 the following elements form a geodesic basis:
\[
X_0,\ X_1\pm \sqrt{\frac32}X_2,\
\pm \sqrt{2} X_3+X_4.
\]
 For $\g_{19}(\alpha)$, we consider two subcases. If $\alpha\geq 0$,  the following elements form a geodesic basis:
\[
X_0,\ \sqrt{2(1+\alpha)}X_1+X_4,\
\sqrt{2(1+\alpha)}X_2+\sqrt{\alpha}X_4,\
\pm \sqrt{2} X_3+X_4.\]
If $\alpha<0$,  the following elements form a geodesic basis:
\[
X_0,\ \pm\sqrt{-\alpha}X_1+X_2,\
\pm \sqrt{2} X_3+X_4.\]

\subsection{$\n\cong \m_0(4)$.} 
In \cite[Example 1]{CLNN}, an example is given of a 5-dimensional unimodular Lie algebra with trivial centre with $\n\cong \m_0(4)$. As shown in \cite{CLNN}, this algebra has an inner product with a geodesic basis but it doesn't have an inner product with an orthonormal geodesic basis. We will prove that up to isomorphism, the algebra of \cite{CLNN} is the only 5-dimensional unimodular Lie algebra with trivial centre with $\n\cong \m_0(4)$.

Let $\g$ be such an algebra. Choose a basis $\{X_0,X_1,\dots, X_4\}$  for $\g$ so that $\n=\Span(X_1,\dots, X_4)$ and $[X_1,X_i]=X_{i+1}$ for $i=2,3$. Let $\varphi=\ad(X_0)|_{\n}$. Since the centre $\Span(X_4)$  of $\n$ and the derived algebra $\Span(X_3,X_4)$ of $\n$ are characteristic ideals of $\n$, they are invariant under the derivation $\varphi$. So the matrix representation of $\varphi$ relative to the basis $\{X_1,\dots, X_4\}$ has the form
\begin{equation} \label{mrep2}
\begin{pmatrix}
a & b & 0 & 0 \\
c & d & 0 & 0 \\
\alpha  & \beta  &  e   &   0     \\
\gamma & \delta      & f    &\lambda 
\end{pmatrix}.
\end{equation}
Notice that as $\varphi$ is a derivation, applying $\varphi$ to the relations $[X_1,X_2]=X_3, [X_2,X_3]=0,[X_1,X_3]=X_4$ gives respectively
\begin{equation} \label{deri}
a+d=e,\quad b=0,\quad a+e=\lambda,\quad\text{and}\quad f=\beta.
\end{equation}
Then, as $\varphi$ has zero trace, $4a+3d=0$. Since, by assumption, $\g$ has trivial centre, we have $\lambda\not=0$. 
So by rescaling $\varphi$ if necessary we may assume that  
$\varphi$ has the matrix representation 
\begin{equation} \label{mrep3}
\begin{pmatrix}
3 &0 & 0 & 0 \\
c & -4 & 0 & 0 \\
\alpha  & \beta  &  -1   &   0     \\
\gamma & \delta      & \beta    &2
\end{pmatrix}.
\end{equation}

Notice that $\varphi$ has 4 distinct real eigenvalues, $3,-4,-1,2$ and so it has corresponding eigenvectors $Y_1,\dots,Y_4$. With respect to the basis $\{Y_1,\dots,Y_4\}$ for $\n$, the  matrix representation of $\varphi$ is diagonal.
Note that as $\varphi$ is a derivation,
\[
\varphi([Y_1,Y_2])=[\varphi(Y_1),Y_2]+[Y_1,\varphi(Y_2)]=[3 Y_1,Y_2]+[Y_1,-4Y_2]=-[ Y_1,Y_2].\]
That is, $[Y_1,Y_2]$ is either zero or it is an eigenvector of $\varphi$ with eigenvalue $-1$. Hence $[Y_1,Y_2]$ is multiple of $Y_3$, say $[Y_1,Y_2]=\mu Y_3$. Similarly, $[Y_1,Y_3]$ is multiple of $Y_4$, say $[Y_1,Y_3]=\nu Y_4$. By the same reasoning, since $\varphi$ does not have $5,-5,-2$ or $1$ as an eigenvalue, $[Y_1,Y_4]=[Y_2,Y_3]=[Y_2,Y_4]=[Y_3,Y_4]=0$.
Then, since $\n$ has an element of maximal nilpotency, $\mu$ and $\nu$ must both be nonzero. By rescaling $Y_2$ and $Y_3$, we may take $\mu=\nu=1$.
The basis $\{X_0,Y_1,\dots,Y_4\}$ now has the same relations as the algebra of \cite[Example 1]{CLNN}.
So $\g$ is isomorphic to this algebra.

\subsection{$\n\cong \H_3$.} 
We will show that in this case the centre of $\g$ is nontrivial, contrary to our assumption.
Consider a basis $\{X_1,\dots, X_5\}$ for $\g$ where  $\n=\Span(X_3,X_4,X_5)$ and  $[X_3,X_4]=X_5$. Let $\varphi$ be a  derivation of $\n$ with zero trace. Since $\Span(X_5)$ is the centre of $\n$, the matrix representation of $\varphi$ relative to the basis $\{X_3,X_4, X_5\}$ has the form
\begin{equation} 
A=\begin{pmatrix}
a & b & 0 \\
c & d & 0  \\
e  & f  & g
\end{pmatrix}
\end{equation}
where $a+d+g=0$, and as $\varphi$ is a  derivation, $a+d=g$. Hence $g=0$ and so $\varphi(X_5)=0$. 
In particular $[X_1,X_5]=[X_2,X_5]=0$, and hence $X_5$ lies in the centre of $\g$. 

\subsection{$\n\cong \R^3$.} Consider a basis $\{X_1,\dots, X_5\}$ for $\g$ where  $\n=\Span(X_3,X_4,X_5)$. According to \cite{Mub1,Mub2}, for unimodular Lie algebras with trivial centre, when $\n\cong \R^3$, we may choose our basis so that $[X_1,X_2]=0$ and the relations may be taken to be one of the following two cases, relative to the basis $\{X_3,X_4,X_5\}$ for $\n$:
\begin{enumerate}
\item $\g_{33}$:\qquad $\ad(X_1)|_{\n}=\begin{pmatrix} 1 & 0 & 0 \\ 0 &  0 &0  \\ 0 & 0 & -1 \end{pmatrix}$, $\ad(X_2)|_{\n}=\begin{pmatrix} 0& 0 & 0 \\ 0 &  1 &0  \\ 0 & 0 & -1 \end{pmatrix}$, 
    \item $\g_{35}$:\qquad $\ad(X_1)|_{\n}=\begin{pmatrix} -2 & 0 & 0 \\ 0 & 1&0  \\ 0 & 0 & 1 \end{pmatrix}$, $\ad(X_2)|_{\n}=\begin{pmatrix} 0& 0 & 0 \\ 0 &  0 &1  \\ 0 & -1 & 0 \end{pmatrix}$.
\end{enumerate}

 For $\g_{35}$, consider  the inner product for which the basis $\{X_1,  X_2, Y_3=X_3+X_4, Y_4= X_3 - \frac12 X_4 + \frac{\sqrt{3}}2  X_5, Y_5= X_3 - \frac12 X_4 - \frac{\sqrt{3}}2  X_5\}$ is orthonormal. We have
 \begin{align*}
 [X_1,Y_3]&=-2X_3+X_4=-Y_4-Y_5,\\
 [X_1,Y_4]&=-2X_3-\frac12X_4+\frac{\sqrt{3}}2  X_5=-Y_3-Y_5,\\
[X_1,Y_5]&=-2X_3-\frac12X_4-\frac{\sqrt{3}}2  X_5=-Y_3-Y_4,\\
 [X_2,Y_3]&=-X_5=\frac1{\sqrt3}(-Y_4+Y_5),\\
 [X_2,Y_4]&=\frac{\sqrt{3}}2  X_4+\frac12X_5=\frac1{\sqrt3}(Y_3-Y_5),\\
[X_2,Y_5]&=-\frac{\sqrt{3}}2  X_4-\frac12X_5=\frac1{\sqrt3}(-Y_3-Y_4).
 \end{align*}
 So this basis is geodesic.
 
Since the algebra $\g_{33}$ has an abelian nilradical $\n$,  it has an inner product with a geodesic basis, by Theorem \ref{th:nilabel}.
It remains to prove that  $\g_{33}$ does not admit an inner product with an orthonormal geodesic basis. Suppose by way of contradiction, that it does admit an orthonormal geodesic basis $\mathcal B=\{Y_1,\dots,Y_5\}$.

First suppose that three of the basis elements belong to $\g'$, say $\g'=\Span(Y_3,Y_4,Y_5)$. Relative to this basis, let $A=(A_{ij}),B=(B_{ij})$ denote the matrix representations of the maps $\ad(X_1)|_{\g'},\ad(X_2)|_{\g'}$ respectively. Note that as the basis elements are geodesic, $A,B$ must have zeros on the diagonal. As $[A,B]=0$, an easy calculation shows that the three two-dimensional vectors $(A_{12}, B_{12}), (A_{23}, B_{23})$ and $(A_{31}, B_{31})$ are collinear. It follows that a certain nontrivial linear combination $C$ of $A$ and $B$ has $C_{12}=C_{23}=C_{31}=0$, and all zeros on the diagonal. The eigenvalues of such a matrix are $r, r \omega$ and $r \omega^2$, where $r \in \mathbb{R}$ and $\omega$ is a nonreal cubic root of unity. This is a contradiction, as any nontrivial linear combination of $A$ and $B$ has real eigenvalues, which are not all zeros.
As the orthogonal complement $\g'^\perp$ is 2-dimensional, it cannot contain two of the elements of $\mathcal B$ or the other three would be in $\g'$. So we have at most two basis elements in $\g'$ and at most one basis element in $\g'^\perp$.

Consider a geodesic unit vector of the form $Y= \sum_{i=1}^5 a_i X_i$ with $Y\not\in\g'$.
If  $a_1,a_2$ and $a_1+a_2$ all nonzero, then the image of $\ad(Y)$ is $\g'$, and so $Y$ is in  $\g'^\perp$.
Now suppose that $Y\not\in\g'$ and $Y\not\in\g'^\perp$. So either $a_1,a_2$ or $a_1+a_2$ is zero.
If $a_2=0$, then $a_4=0$, as otherwise the image of $\ad(Y)$ would be $\g'$ and we would have $Y\in\g'^\perp$.
Note that as $[Y,X_3]=a_1X_3, [Y,X_5]=-a_1X_5$, we have $Y \in \Span(X_1,X_3,X_5)$ and $Y$ is orthogonal to $\Span(X_3,X_5)$. Up to the sign, there is only one unit vector, $Z_1$ say, in $\Span(X_1,X_3,X_5) \cap (\Span(X_3,X_5))^\perp$, so  $Y=\pm Z_1$. Similarly, if $a_1=0$, then $Y=\pm Z_2$, where $Z_2$ is a unit vector in $\Span(X_2,X_4,X_5) \cap (\Span(X_4,X_5))^\perp$. If  $a_1+a_2=0$, then $Y=\pm Z_3$, where $Z_3$ is a unit vector in  $\Span(X_1-X_2,X_3,X_4) \cap (\Span(X_3,X_4))^\perp$.  So we have shown that if $Y$ is a geodesic unit vector and $Y\not\in\g'\cup\g'^\perp$, then $ Y\in\pm \{Z_1,Z_2,Z_3\}$. Notice however that we have not yet excluded the possibility that some of the elements $Z_1,Z_2,Z_3$ belong to $\g'^\perp$.  In summary, so far we have that, of the five elements of $\mathcal B$, at most  two are in $\g'$, at most one is in $\g'^\perp$ and, up to signs, the rest belong to $\{Z_1,Z_2,Z_3\}$.
So at least two vectors in $\pm\{Z_1,Z_2,Z_3\}$ belong to $\mathcal B$.

Let $W_1=\Span(X_3,X_5), W_2=\Span(X_4,X_5), W_3=\Span(X_3,X_4)$. So $Z_i$ is orthogonal to $W_i$, for $i=1,2,3$. Note that none of the elements $X_3,X_4,X_5$ are geodesic, since they are each eigenvectors of $\ad(X_1)$ or $\ad(X_2)$. So
 if one of the elements of $\mathcal B$, say $Y_5$, lies in $ \g'$, then $Y_5$ lies in the complement of two of the spaces $W_i$; suppose, for example, that $Y_5\in \g'\backslash(W_1\cup W_2)$. Then $Z_1$ is orthogonal to both $W_1$ and $Y_5$ and hence $Z_1\in \g'^\perp$, and by the same reasoning, $Z_2\in \g'^\perp$. So we conclude that one of the elements of $\mathcal B$ in $\pm\{Z_1,Z_2,Z_3\}$ is in $\g'^\perp$. Hence, since at most one one of the elements of $\mathcal B$ is in $\g'^\perp$,  exactly  two of the elements of $\mathcal B$ are in $\g'$, and all the elements  $\pm\{Z_1,Z_2,Z_3\}$ belong to $\mathcal B$. But then, by the argument we just used, two of these latter elements are in 
$\g'^\perp$, which is a contradiction.

This concludes the proof of Theorem \ref{T:dim5}.

\bibliographystyle{amsplain}
\bibliography{geod}

\end{document}